\documentclass{article}
\usepackage{amsmath}
\usepackage{amsthm}
\usepackage{amsfonts}
\usepackage{amssymb}
\usepackage{graphicx}
\usepackage{color}
\usepackage{mathrsfs}

\newtheorem{definition}{Definition}
\newtheorem{thm}{Theorem}
\newtheorem{prp}{Proposition}

\newtheorem{remark}{Remark}

\newtheorem{example}{Example}
\allowdisplaybreaks

\numberwithin{equation}{section}

\title{Models of space-time random fields on the sphere}

\author{
Mirko D'Ovidio\footnote{Department of Basic and Applied Sciences for Engineering, Sapienza
University of Rome, Via A. Scarpa, 16, 00161, Roma, Italy},
Enzo Orsingher\footnote{Department of Statistical Sciences, Sapienza University of Rome, P.le Aldo Moro, 5, 00185, Rome, Italy},
Lyudmyla Sakhno\footnote{Taras Shevchenko National University of Kyiv, Ukraine}
}

\date{}

\begin{document}

\maketitle

\begin{abstract}
We study general models of random fields associated with  non-local equations in time and space. We discuss the properties of the corresponding angular power spectrum and find asymptotic results in terms of random time changes.
\end{abstract}

{\bf Keywords:}  fractional equations, spherical Brownian motion, subordinators, random fields, Laplace-Beltrami operators, spherical harmonics.

{\bf AMS MSC 2010:} 60G60; 60G22; 60H99

\section{Introduction}

The models of spherical random fields are in great demand in various applied areas such as geophysics, geodesy, planetary sciences, astronomy, cosmology and others. 
In recent years one can observe the growing popularity of stochastic partial differential equations   for modeling space-time random fields. Solutions to stochastic Cauchy problems for various classes of partial differential equations on the sphere admit exact series representations, which is important, in particular, for numerical approximation of such random fields, since this can be achieved effectively by truncating the corresponding expansions (see, e.g., \cite{ABOW} and references therein).

Papers \cite{D}, \cite{Dov} develop the approach to construct time dependent random fields on the sphere through coordinates change and subordination. These models of random fields arise as solutions to partial differential equations with operators of a particular form and random initial condition represented by a Gaussian random field. 

In the present paper we generalize results of paper \cite{D} and consider  random fields arising  as
solutions to the fractional equations of the form
\begin{equation}
\left( \gamma -\Psi(-\Delta_{\mathbb{S}^2_1}) +\mathfrak{D}^\Phi_t\right) X_{t}(x)=0, \,  x\in \mathbb{S}_{1}^{2},\;t>0,\;\gamma \ge0
 \label{eq1-intro}
\end{equation}%
subject to the initial condition $X_{0}(x)=T(x)$, with $T(x), x\in \mathbb{S}_{1}^{2}$, being a square integrable isotropic Gaussian random field.

In the above equation,  the generalized Laplace-Beltrami operator
$\Psi(-\Delta_{\mathbb{S}^2_1})$ is defined in terms of the
transition semigroup of the subordinate rotational Brownian motion
$B^\Psi_t=B_{H_t}$, where $H_t$ is a subordinator with the Laplace
exponent $\Psi$, and this covers, in particular, the case of
fractional Laplace operator $(-\Delta_{\mathbb{S}^2_1})^\alpha$. The time
derivative $\mathfrak{D}^\Phi_t$ is the generalized fractional convolution-type derivative associated with the Bern\v{s}tein
function $\Phi$, which reduces to 
  the
Caputo-Djrbashian fractional derivative $\frac{\partial ^{\beta }}{\partial t^{\beta }}$ for $\Phi(\lambda)=\lambda^\beta$. 

Note that the treatment of the generalized Laplace-Beltrami
operator $\Psi(-\Delta_{\mathbb{S}^2_1})$, where $\Psi$ is some
Bern\v{s}tein function, by means of the transition semigroup of the
subordinate rotational Brownian motion $B^\Psi_t$ gives the
possibility of a deeper insight into the structure of the solutions
$X_{t}(x)$ of the equation \eqref{eq1-intro}, and permits us to obtain
not only its Karhunen-Lo\`{e}ve expansion, but also its
representation as a coordinate changed random field.

It was 
shown in \cite{T} that convolution-type derivatives provide the unifying framework for the study of
subordinators and their inverse processes, and, in particular, the governing equations for
densities of subordinators and their inverses are obtained in terms of the convolution type
derivatives. The introduction of these derivatives has also inspired numerous recent
studies of new types of equations suitable to describe anomalous diffusion and other
complex processes.

We show that solution to \eqref{eq1-intro} is a time-varying random field with the following representation in terms of spherical harmonics $Y_{lm}$:
\begin{equation}
X_{t}(x)=\sum_{l=0}^{\infty }\sum_{m=-l}^{+l} a_{lm}   \,
\widetilde{l}(t, \gamma +\Psi (\mu _{l}))
 Y_{lm}(x), \,  x\in \mathbb{S}_{1}^{2},\;t>0,
\label{intr2}\end{equation}
where
\begin{equation*}
a_{lm}=\int_{\mathbb{S}_{1}^{2}}X_{0}(x) Y_{lm}^{\ast }(x)\lambda(dx),
\end{equation*}
and $\widetilde{l}$ is associated with the function $\Phi$, namely, $\widetilde{l}(t, \lambda)=\mathbf{E}[\exp( - \lambda L_t)]$ with $L$ being the inverse process for the subordinator with the Laplace exponent $\Phi$. Besides the Karhunen-Lo\`{e}ve expansion \eqref{intr2}, we also represent the solution as a coordinate changed random field.

The use of the generalized derivative $\mathfrak{D}^\Phi_t$ allows to constract more general models of random fields that those in paper \cite{D}, where the fractional Caputo-Djrbashian  derivative was used.

The paper is organized as follows. Sections \ref{S2}--\ref{S4} make necessary preparations and provide a concise review on the operators used in equations and facts on isotropic random fields. The main results are stated in Section \ref{S6}: we give the different representations for solutions to \eqref{eq1-intro} and discuss their properties.

\section{Generalized fractional operators}\label{S2}

To define our models of space-time random fileds, we will use partial differential equations with generalized fractional derivatives in time and space variables. 

\subsection{Generalized fractional Caputo-Djrbashian   or convolution-type derivative}

We first introduce the generalized fractional operator to act on the time variable.

Let us consider the subordinator $H$, that is, a non negative
L\'evy process with almost surely increasing paths.  
The process $H$ is
characterized by a L\'evy measure $\Pi$ on $(0, \infty)$ such that 
$\int_0^\infty (1 \wedge z) \Pi(dz)< \infty$ and the
corresponding Bernst\v{e}in function $\Phi$ (called the Laplace exponent or symbol of $H$). That is,
\begin{align*}
\mathbf{E}[\exp( - \lambda H_t)] = \exp(- t \Phi(\lambda))
\end{align*}
where
\begin{equation}\label{LevKinFormula}
\Phi(\lambda) = \int_0^\infty \left( 1 - e^{ - \lambda z} \right) \Pi(dz), \quad \lambda \geq 0. 
\end{equation}

In the general case, the expression for Bernst\v{e}in function \eqref{LevKinFormula} contains two more terms, namely, of the form  $a+b\lambda$, but we consider now the case $a=b=0$.

We also recall that
\begin{align}
\label{tailSymb}
\frac{\Phi(\lambda)}{\lambda} = \int_0^\infty e^{-\lambda z} \overline{\Pi}(z)dz, \qquad \overline{\Pi}(z) = \Pi((z, \infty))
\end{align}
and $\overline{\Pi}$ is the so called \emph{tail of the L\'evy measure}. For details, see the book  \cite{BerBook}.

Introduce
the inverse process associated to $H$ (and, so, associated to $\Phi$) as
$$L_t = \inf \{s \geq 0\,:\, H_s >t \}, \,\,\, t>0.$$ 
$L$
is a non negative process with almost surely non decreasing paths.

 We assume that $\Pi((0, \infty)) = \infty$ and, therefore, we focus only on strictly increasing subordinators. For this case, the inverse process $L$ turns out to be a continuous process. 
 Moreover, under the additional assumption that the tail $\overline{\Pi}(z)$, $z\ge 0$, is absolutely continous function,  the inverse process $L_t$ possesses the probability density function $l(s,t)$ for each $t>0$ (see, \cite{T}).

For the reader's convenience we recall symbols for some important classes of subordinators relevant to our study:
\begin{itemize}
\item $\Phi(\mu) = \mu^\alpha$: stable, $\Pi(y) = \alpha y^{-\alpha -1}/ \Gamma(1-\alpha)$;
\item $\Phi(\mu) = b\mu + \mu^\alpha$: stable with drift, $\Pi(\cdot)$ as above with the drift coefficient $b>0$;
\item $\Phi(\mu) = (\mu+\beta)^\alpha-\beta^\alpha$: tempered stable, $\Pi(y) = \alpha e^{-\beta y} y^{-\alpha -1}/ \Gamma(1-\alpha)$;
\item $\Phi(\mu) = \ln (1 + \mu )$: gamma, $\Pi(y) = y^{-1}e^{-y}$;
\item $\Phi(\mu) = \ln (1+ \mu^\alpha)$: geometric stable, $\Pi(y)=\alpha y^{-1}E_\alpha(-y)$ where
\begin{equation*}
E_\alpha(z) = \sum_{j=0}^\infty \frac{z^j}{\Gamma(\alpha j + 1)}
\end{equation*}
is the Mittag-Leffler function.
\end{itemize}

\begin{definition} Convolution-type derivative associated with the function $\Phi$ given by \eqref{LevKinFormula} is defined for the absolutely continuous function $u$ by the formula
\begin{align}
\label{D2} \mathfrak{D}^\Phi_t u(t) = 
\int_0^t \frac{\partial}{\partial t}u(t-s)\overline{\Pi}(s)ds.
\end{align}
\end{definition}
According to the definition, the  generalized fractional operator is characterized by the Bernst\v{e}in function $\Phi$. Thus, such operator can be associated with the processes $H$ and $L$ previously introduced. Moreover, it was shown \cite{T} that this operator can be used to study the properties of subordinators and their inverses in the unifying manner  and write the governing equations for their densities.

We notice that when $\Phi(\lambda)=\lambda$ (that is we deal with the ordinary derivative) we have that $H_t = t$ and $L_t=t$ almost surely.

In the case where
$\Phi(\lambda)=\lambda^{\alpha}$, $\alpha\in (0,1)$, we have that
\begin{equation*}
\mathfrak{D}^\Phi_t u(t) = \frac{d^{\alpha}}{dt^{\alpha}}u(t)=\frac{1}{\Gamma(1-\alpha)}\int_{0}^{t}\frac{u^\prime(s)}{\left(t-s\right)^{\alpha}}ds,
\end{equation*}
where $u^\prime= du/ds$, that is the well-known Caputo-Djrbashian fractional derivative.

Similarly to the Caputo-Djrbashian fractional derivative, the convolution type derivative can be characterized (and alternatively defined) by means of its Laplace transform.

Let $M>0$ and $w\geq 0$. Let $\mathcal{M}_w$ be the set of (piecewise) continuous functions on $[0, \infty)$ of exponential order $w$ such that $|u(t)| \leq M e^{wt}$. Denote by $\widetilde{u}$ the Laplace transform of $u$. Then, we define the operator $\mathfrak{D}^\Phi_t : \mathcal{M}_w \mapsto \mathcal{M}_w$ such that
\begin{align}
\label{PhiConvDef}
\int_0^\infty e^{-\lambda t} \mathfrak{D}^\Phi_t u(t)\, dt = \Phi(\lambda) \widetilde{u}(\lambda) - \frac{\Phi(\lambda)}{\lambda} u(0), \quad \lambda > w
\end{align}
where $\Phi$ is given in \eqref{LevKinFormula}. Since $u$ is exponentially bounded, the integral $\widetilde{u}$ is absolutely convergent for $\lambda>w$.  By Lerch's theorem the inverse Laplace transforms $u$ and $\mathfrak{D}^\Phi_tu$ are uniquely defined. Formula \eqref{PhiConvDef} can be rewritten as follows
\begin{align}
\label{PhiConv}
\Phi(\lambda) \widetilde{u}(\lambda) - \frac{\Phi(\lambda)}{\lambda} u(0) = & \left( \lambda \widetilde{u}(\lambda) - u(0) \right) \frac{\Phi(\lambda)}{\lambda}
\end{align}
and thus, $\mathfrak{D}^\Phi_t$ can be regarded as a convolution involving the ordinary derivative and the inverse transform of \eqref{tailSymb} iff $u \in \mathcal{M}_w \cap C([0, \infty), \mathbb{R}_+)$ and $u^\prime \in \mathcal{M}_w$.

The operator $\mathfrak{D}^\Phi_t$ have been introduced and
studied in the papers \cite{K}, \cite{Chen17},
\cite{T}. 

\medskip
In Section \ref{S6} we study random fields on the sphere governed by equations with convolution-type derivatives $\mathfrak{D}^\Phi_t$. The following well-known fact will be important further on. Thus, we state it properly.

\begin{prp}
\label{prUseful}
Let $L$ be the inverse process for a subordinator with
Bern\v{s}tein function $\Phi$, and assume that $\Pi(0,\infty)=\infty$ and the
tail $\overline\Pi(s)=\Pi(s,\infty)$ is absolutely continuous. For the 
process $L_t$, $t>0$ we have
\begin{align}\label{l}
l(t,x) = \mathbf{P}(L_t \in dx)/dx, \quad t,x>0
\end{align}
with  Laplace transform
\begin{equation}\label{tilde_l}
\tilde{l}(t,\lambda)=\int_{0}^{\infty}e^{-\lambda x} l(t,x)dx=\mathbf{E} [e^{-\lambda L(t)}]
\end{equation}
which satisfies the following equation
\begin{equation}\mathfrak{D}^\Phi_t\tilde{l}(t,\lambda)=-\lambda \tilde{l}(t,\lambda),\label{lapl}\end{equation}
and thus, $\tilde{l}(t,\lambda)$ is an eigenfunction of the operator $\mathfrak{D}^\Phi_t$ corresponding to the eigenvalue $\lambda$.
\end{prp}
\begin{remark}
For $L$ being the inverse process for stable subordinator, that is, $\Phi(\lambda)=\lambda^\alpha$, \eqref{lapl} reduces to the well known fact that the Mittag-Leffler function is an eigenfunction of the Caputo-Djrbashian fractional derivative:
\begin{equation*}\frac{\partial^\alpha}{\partial t^\alpha}E_\alpha(-t^\alpha\lambda)=-\lambda E_\alpha(-t^\alpha\lambda).\end{equation*}
\end{remark}
\begin{remark}
The proof of the above
Proposition has been derived, by different approaches, in \cite{K}, \cite{T}, \cite{B}. 
In particular, 
in \cite{B}, the authors consider the time-changed Poisson process $N_{L_t}$ (with $L_t$ being an inverse subordinator independet of the process $N$) and show that its marginal distributions
    $p_x(t)=\mathbf{P}(N_{L_t}=x)$, $x=0,1,\ldots$, 
    satisfy the difference-differential equations
    \begin{equation*}
        \mathfrak{D}^\Phi_t
        p_x(t)=-\lambda\left[p_x(t)-p_{x-1}(t)\right]
        \end{equation*}
with initial condition $p_x(0)=1, x=0$ and $p_x(0)=0, x\ge 1$.
        Then \eqref{lapl} is deduced as a consequence of the above equation, since $p_0(t)=\mathbf{P}(N_{L_t}=0)=\int_0^\infty e^{-\lambda x} {l}(t,x)dx =\tilde{l}(t,\lambda)$.
\end{remark}

\subsection{Generalized fractional Laplacian on the sphere}\label{subsec2.2}

Let  $f \in L^2(\mathbb{S}^2_1)=L^2(\mathbb{S}^2_1, \mu)$, where $\mu$ is the Lebesgue measure on the unit sphere $\mathbb{S}^2_1$: 
$$\mu(dx) = \mu(d\vartheta, d\varphi) = d\varphi\, d\vartheta\, \sin \vartheta$$ with $x \in \mathbb{S}^2_1$ being represented as
\begin{equation*}
x = (\sin \vartheta \cos \varphi, \sin \vartheta \sin \varphi, \cos \vartheta), \quad \vartheta \in  [0, \pi], \; \varphi \in [0, 2\pi).
\end{equation*} 

The set of spherical harmonics $\{Y_{lm}:\, l \geq 0,\; m=-l, \ldots , +l\}$ represents an orthogonal basis for the space $L^2(\mathbb{S}^2_1)$.
 
Recall that for a fixed integer $l$  the spherical harmonics
\begin{equation*}
Y_{lm}(\vartheta, \varphi) = \sqrt{\frac{2l+1}{4 \pi} \frac{(l-m)!}{(l+m)!}} Q_{lm}(\cos \vartheta) e^{im\varphi}
\end{equation*}
(or linear combination of them) solve the eigenvalue problem
\begin{equation}
\triangle_{\mathbb{S}_{1}^2} Y_{lm}= - \mu_l \, Y_{lm} \label{eigenY}, \quad l \geq 0, \; |m| \leq l
\end{equation}
where, the eigenvalues are given by
$$\mu_l=l(l+1)$$
and the operator
\begin{align}
\triangle_{\mathbb{S}_{1}^2} = &  \frac{1}{\sin^2 \vartheta} \frac{\partial^2}{\partial \varphi^2} + \frac{1}{\sin \vartheta} \frac{\partial}{\partial \vartheta} \left( \sin \vartheta \frac{\partial}{\partial \vartheta} \right) , \quad \vartheta \in [0,\pi],\; \varphi \in [0, 2\pi), \label{spherical-laplace}
\end{align}
is the spherical Laplace operator (called also Laplace-Beltrami
operator). The spherical
harmonics are written in terms of
\begin{equation*}
Q_{lm}(z)=(-1)^m (1-z^2)^{m/2}\frac{d^m}{d z^m}Q_{l}(z)
\end{equation*}
which are the associated Legendre functions and the Rodrigues' formula
$$ Q_l(z) = \frac{1}{2^l l!}\frac{d^l}{dz^l} (z^2 - 1)^l $$
defines the Legendre polynomials. 

For $f \in L^2(\mathbb{S}^2_1)$ we have the representation
\begin{align*}
f(x) = \sum_{l=0}^\infty \sum_{m=-l}^l f_{lm} Y_{lm}(x), \quad x \in \mathbb{S}^2_1,
\end{align*}
which holds in the $L^2$ sense, where
\begin{equation}
f_{lm}=\int_{\mathbb{S}^2_1} f(x) Y_{lm}^*(x)\,dx=\int_{0}^{2\pi} \int_0^\pi f(\vartheta, \varphi) Y^*_{lm}(\vartheta, \varphi)\, \sin \vartheta\, d\vartheta\, d\varphi, \quad |m|\leq l, \; l=0,1,2,\ldots  . \label{f-coeff-exp}
\end{equation}
(see, for example, the Peter-Weyl representation theorem on the sphere in \cite{DomPec-book} and references therein).

The {angular power spectrum} of $f$ is defined as
\begin{equation}\label{def-angular-spectrum}
f_l  = \sum_{|m|\leq l } | f_{lm}|^2=  \sum_{|m| \leq l} \Bigg| \int_{\mathbb{S}^2_1} f(x) Y^*_{lm}(x)\mu(dx) \Bigg|^2, \quad l=0,1,2, \ldots .
\end{equation}

We next define the generalized fractional Laplace  operators on the sphere following \cite{Dov}, \cite{D}.
Let $F(t)$, $t\ge 0$ be a L\'evy subordinator with  the
 Laplace exponent $$\Psi(\lambda)  = b\lambda + \int_0^\infty \left( 1 - e^{ - \lambda z} \right) M(dz),\,\, b\ge0,\,\, \lambda \geq 0,$$ with $M$ being the corresponding  L\'evy measure.

Let $B_{t}$, $t\ge 0$, be a Brownian motion on the unit sphere $\mathbb{S}^2_1$.  Its transition density can be writtes as follows (see 
\cite{Yios1949}): 
\begin{eqnarray}
Pr\{x+B_{t}\in dy\}/dy &=&Pr\{B_{t}\in dy\,|\,B_{0}=x\}/dy  
=\sum_{l=0}^{\infty }\sum_{m=-l}^{+l}e^{-t\mu _{l}}\mathcal{Y}_{l,m}(y)%
\mathcal{Y}_{l,m}^{\ast }(x).  
\label{PdiB}
\end{eqnarray}%
Consider the the initial-value problem 
\begin{equation}
\left\{ 
\begin{array}{ll}
\displaystyle\frac{\partial u}{\partial t}=\Delta _{\mathbf{S}_{1}^{2}}u, & 
x\in \mathbf{S}_{1}^{2},\,t>0 \\ 
\displaystyle u(x,0)=f(x) & 
\end{array}%
\right. 
\end{equation}%
for  $f\in L^2(\mathbf{S}_{1}^{2}).$
The solution to the above problem can be written as follows:
\begin{equation}
u(x,t)=P_{t}f(x)=\mathbb{E}f(x+B_{t})=\int_{\mathbf{S}_{1}^{2}}f(y)Pr\{x+B_{t}\in
dy\} =\sum_{l=0}^{\infty }\sum_{m=-l}^{+l}e^{-t\mu _{l}}\mathcal{Y}_{l,m}(x)f_{lm},
 \label{semigB}
\end{equation}%
that is, the solution is given by the transition semigroup of the rotational Brownian motion $B_{t}$,  $t>0$, with values in $\mathbf{S}_{1}^{2}$.

In \cite{Dov} the following operator acting on $f \in
L^2(\mathbf{S}^2_1)$ was introduced:
\begin{equation}
\Psi(-\Delta _{\mathbf{S}_{1}^{2}}) f(x) := \int_0^\infty \left( P_t\, f(x) - f(x) \right) M(dt).
\label{frac-oper-sphere}
\end{equation}
It was shown in  \cite{Dov} (see also \cite{D}) that 
\begin{equation}
\Psi(-\Delta _{\mathbf{S}_{1}^{2}}) Y_{lm}(x) = -\Psi(\mu_l)Y_{lm}(x),
\label{eigenPsi}
\end{equation}
and thus the spherical harmonics are the eigenfunctions of the operator $\Psi(-\Delta _{\mathbf{S}_{1}^{2}})$ with the eigenvalues $-\Psi(\mu_l)$.This fact was stated by direct calculations using the
semigroup approach and the spectral representation of the operator $\Psi(-\Delta)$ (or Phillips representation) in \eqref{frac-oper-sphere}.

Basing on \eqref{eigenPsi}, the action of the operator $\Psi(-\Delta _{\mathbf{S}_{1}^{2}})$ can be also defined by means of a series representation as given below.

Let us consider the space of functions
\begin{equation}
H^{s}(\mathbb{S}^2_1) = \left\lbrace f \in L^2(\mathbb{S}^2_1):\,  \sum_{l=0}^{\infty} (2l+1)^{2s} f_l < \infty  \right\rbrace \label{SobolevSp}
\end{equation}
where
$f_l$ is the angular spectrum of $f$ (see \eqref{def-angular-spectrum}).

\begin{definition}
Let $f \in H^s(\mathbb{S}^2_1)$ and $s>5/4$. Then 
\begin{align}
\Psi(-\Delta_{\mathbf{S}_{1}^{2}}) f(x) := \sum_{l=0}^\infty \sum_{m=-l}^{+l} f_{lm} Y_{lm}(x) \Psi(\mu_l). \label{def-der-psi-gen}
\end{align}
\end{definition}
Note that since $\Psi$ is the symbol of a subordinator, as $l \to \infty$, we have that $\Psi(l)/l \to 0$ (we write $\Psi(l) < l$ for large $l$) whereas, as $l \to 0$, $\Psi(l)\to 0$. The series in \eqref{def-der-psi-gen}
converges absolutely and uniformly. This can be proved by considering that $f_l < l^{-2s}$ with $s> 5/4$  (indeed, $f \in H^s(\mathbb{S}^2_1)$) and for the harmonic eigenfunction we have that $\| Y_{lm}\|_{\infty} < l^{1/2}$ (see \cite{Quantum}). Since $\Psi(\mu_l) < l^2$ we have the claimed convergence. We refer for more details to \cite{D}, \cite{Dov}.

\section{Isotropic random fields on the unit-radius sphere}\label{S4}

Let us consider a real-valued, zero-mean, isotropic Gaussian random field $T(x)$, $x\in \mathbb{S}^2_1$, that is, we assume 
$\mathbb{E}T(x)=0$, $\mathbb{E}T^2(x)<\infty$, and for any $g\in SO(3)$ (the special group of rotations in $\mathbb{R}^{3}$) we have: $\mathbb{E}T(gx_1)T(gx_2)=\mathbb{E}T(x_1)T(x_2)$,  $x, x_1, x_2\in \mathbb{S}^2_1$.

For the field $T$ we can write the
 spectral representation
\begin{equation}
T(x) = \sum_{l=0}^{\infty} \sum_{m=-l}^{+l} a_{lm}Y_{lm}(x) = \sum_{l=0}^{\infty} T_l(x), \label{Trep}
\end{equation}
where
\begin{equation}
a_{lm}= \int_{\mathbb{S}^2} T(x) Y_{lm}^*(x) \lambda(dx) \label{alm-coeff-intro}
\end{equation}
are Fourier random coefficients, $Y_{lm}(x)$ are spherical
harmonics. Convergence in \eqref{Trep} holds in the mean square sense, both with respect to  $L^{2}(dP\times \mu (dx))$
 and with respect to $L^{2}(dP)$ for fixed $x \in \mathbb{S}^2_1$, $\mu(dx)$ is the Lebesgue measure on the unit sphere $\mathbb{S}^2_1$ (see, e.g., \cite{DomPec-book},\cite{Schoenb}):
 \begin{equation*}
\lim_{L\to \infty }\mathbf{E}\Big\| T(x) - \sum_{l=0}^{L}
\sum_{m=-l}^{+l} a_{lm} Y_{lm}(x)\Big\|^2_{L^{2}(\mathbb{S}^2_1)} =\lim_{L\to \infty }\mathbf{E}\Big[ \int_{\mathbf{S}^{2}}\big( T(x) - \sum_{l=0}^{L} \sum_{m=-l}^{+l} a_{lm} Y_{lm}(x) \big)^{2} \mu (dx) \Big] =0
\end{equation*}
and 
\begin{equation*}
\lim_{L\rightarrow \infty }\mathbf{E}\left(T(x)-\sum_{l=0}^{L}\sum_{m=-l}^{+l}a_{lm}Y_{lm}(x)\right)^{2}= 0.
\end{equation*}

\begin{remark}The representation \eqref{Trep} can be deduced as a consequence of the stochastic Peter-Weyl theorem and holds, more generally, for square integrable strictly isotropic random fields, that is, random fields with finite dimensional distributions invariant w.r.t. rotations $g\in SO(3)$: 
$$\{T(x_1), \ldots, T(x_n)\} \stackrel{d}{=} \{T(gx_1), \ldots, T(gx_n)\},$$
where $\overset{d}{=}$ denotes equality in distribution (see, e.g., \cite{DomPec-book}).
\end{remark}

The random coefficients  \eqref{alm-coeff-intro} are zero-mean Gaussian complex random variables such that
\begin{equation}
\mathbf{E}[a_{lm}a^*_{l^\prime m^\prime}] = \delta_l^{l^\prime}\delta_m^{m^\prime} C_l = \delta_l^{l^\prime} \mathbf{E} | a_{lm} |^2 \label{angular-power-C}
\end{equation}
where $C_l$, $l \geq 0$ is the angular power spectrum of the random field $T$ which fully characterizes, under Gaussianity, the dependence structure of $T$. As usual, we denote by
\begin{equation}
\delta_{a}^{b} = \left\lbrace \begin{array}{l}
1, \quad a=b\\
0, \quad a \neq b
\end{array} \right.
\end{equation}
the Kronecker's delta symbol and the symbol ''$^*$'' stands for complex conjugation. We refer to the book by Marinucci and Peccati \cite{DomPec-book} for a deep discussion and presentation of results concerning this field.

In analogy with \eqref{SobolevSp} one can also introduce the space of processes
\begin{align}
\label{spHrond}
\mathcal{H}^s(\mathbb{S}^2_1) = \left\lbrace T \textrm{ as in \eqref{Trep} with } \sum_{lm} \, (\mu_l)^{s}\, \mathbf{E}[|a_{lm}|^2] < \infty \right\rbrace.
\end{align}
Notice that the summability condition in \eqref{spHrond} can be written as
\begin{align}
\label{conSumFrac}
\sum_{l \geq 0}\, (\mu_l)^{s}\, \frac{(2l+1)}{4\pi}\, C_l < \infty
\end{align}
by taking into consideration \eqref{angular-power-C}. We also notice that $\mathcal{H}^2(\mathbb{S}^2_1) \subset L^2(\mathbb{S}^2_1)$ and in particular, the summability condition for $T$ says that 
\begin{align}
\label{condC}
C_l \sim l^{-\theta}, \quad \textrm{with} \quad \theta>2.
\end{align} 
(We use here the usual notation $g \sim f$ meaning that
$
\frac{g(z)}{f(z)} \to 1$ as $z \to \infty.$)

The decay of the angular power spectrum is connected to the smoothness of the covariance. The relation to sample H\"{o}lder continuity and sample differentiability of random fields has been discussed in \cite{LangSch} (see Theorem 4.7). The authors also provided a deep discussion about summability of the angular power spectrum and the formalization in terms of weighted Sobolev space.

Let us take  the isotropic Gaussian random field introduced above as initial condition for the fractional Cauchy problem
\begin{align}\frac{\partial u(t,x)}{\partial t} + \Psi(-\Delta _{\mathbf{S}_{1}^{2}})u(t,x)=0,\label{fracCauchy}\\
u(0,x)=T(x),\notag
\end{align}
where the fractional operator is introduced in section \ref{subsec2.2}, with  $\Psi$ being the Laplace exponent of the subordinator $F$.

In \cite{Dov} it was shown that solution to \eqref{fracCauchy} is given by
\begin{equation*}
u(t,x)=\sum_{l=0}^{\infty }\sum_{m=-l}^{+l}e^{-\Psi(\mu_l)t} a_{lm}   
 Y_{lm}(x)=\mathbf{E}[T(x + B_{F_t}) | \mathfrak{F}_T],
\label{sol-fracCauchy}
\end{equation*}
where $\mathfrak{F}_T$ is the $\sigma$-field generated by $T$ and $B$ is a rotational Brownian motion on the sphere $\mathbb{S}^2_1$ time-changed by the subordinator $F$.

Generalization of \eqref{fracCauchy} by means of the use of the Caputo-Djrbashian fractional derivative in time variable was studied in \cite{D}. In the next section we study the further generalization using the convolution-type derivative defined in section 2.1.

\section{Models of random fields on the sphere}\label{S6}

\subsection{Non-local equations}

Introduce now models of random fields on the sphere driven by equations with fractional operators. We consider fractional operators in time and space associated with  Bern\v{s}tein functions $\Phi$ and $\Psi$ respectively, as defined in Section \ref{S2} above.
 
We suppose that the function $\Phi$ corresponds to the subordinator $H$,   $L$ is its inverse processes possessing the density $l$ with Laplace transform $\widetilde l$ as introduced in Proposition 1 in formulas \eqref{l}--\eqref{tilde_l}. In what follows, we will assume that the conditions of Proposition 1 are valid.

As the initial condition for the fractional equations in the theorems below we consider the isotropic Gaussian random field 
 $T$  defined in \eqref{Trep}.

\begin{thm}
The solution in $L^2(dP \times d\lambda)$ to the fractional equation
\begin{equation}
\left( \gamma - \Psi(-\Delta_{\mathbb{S}^2_1}) + \mathfrak{D}^\Phi_t \right) \,X_{t}(x)=0,\quad x\in \mathbb{S}_{1}^{2},\;t\geq 0, \quad \gamma>0
\label{rand-pde}
\end{equation}%
with initial condition $X_{0}(x)=T(x)$ is a time-dependent random field on
the sphere $\mathbb{S}_{1}^{2}$ written as
\begin{equation}
X_{t}(x)=\sum_{l=0}^{\infty }\sum_{m=-l}^{+l} a_{lm}   \,
\widetilde{l}(t, \gamma +\Psi (\mu _{l}))
 Y_{lm}(x)
\label{X-sol-gammaPDE}
\end{equation}
where
\begin{equation}
a_{lm}=\int_{\mathbb{S}_{1}^{2}}X_{0}(x) Y_{lm}^{\ast }(x)\lambda(dx).
\end{equation}
\label{thm:one}
\end{thm}
\begin{proof} The proof follows the similar lines as those for the proof of Theorem 1 in \cite{D}. In fact, the proof is deduced basing on the common method of separation of variables, the essential component of which is the knowledge of eigenfunctions for the operators involved into the equation. 
 We present the main steps.
 
For the generalized D-C convolution-type derivative we have (see Proposition 1):
\begin{equation}
\mathfrak{D}^\Phi_t \widetilde{l}(t,\lambda)=-\lambda \,
\widetilde{l}(t,\lambda), \quad \lambda>0 \label{eigen-E}
\end{equation}
with $\widetilde{l}(t,\lambda)$ being the Laplace transform of
the inverse subordinator $L_t$ defined in \eqref{tilde_l}.

For the generalized Laplace operator $\Psi(-\Delta)$, we know that
\begin{equation}\label{eigenPsiDelta}
\Psi(-\Delta_{\mathbb{S}^2_1})Y_{lm}(x)=-\Psi(\mu_l)Y_{lm}(x).
\end{equation}
This fact was shown in \cite{D} by direct calculations using the
semigroup approach and the spectral representation \eqref{frac-oper-sphere} of the operator $\Psi(-\Delta)$. Note that, we can also deduce from the result by Dautray
and Lions (see, \cite{DL}, pp. 116-120) that the operator
$\Psi(-\Delta_{\mathbb{S}^2_1})$ has the eigenvalues $\Psi(\mu_l)$ and 
\eqref{eigenPsiDelta} holds.

Thus, assuming that \eqref{X-sol-gammaPDE} holds true, we have that
\begin{align}\label{act1}
\left( \gamma - \Psi(-\Delta_{\mathbb{S}^2_1}) \right) X_t(x) = & \sum_{l=0}^\infty
\sum_{m=-l}^{+l} a_{l,m} \left( \gamma + \Psi(\mu_l) \right) \widetilde{l}(t, \gamma +\Psi (\mu _{l})) Y_{lm}(x).
\end{align}
On the other hand,  using \eqref{eigen-E}, we obtain
\begin{align}\label{act2}
\mathfrak{D}^\Phi_t \,X_{t}(x)
= -\sum_{l=0}^\infty \sum_{m=-l}^{+l} a_{l,m} \left( \gamma +\Psi(\mu_l) \right) \widetilde{l}(t, \gamma +\Psi (\mu _{l}))
Y_{l,m}(x). 
\end{align}
By summing up \eqref{act1} and \eqref{act2}, we obtain  \eqref{rand-pde} as claimed.
\end{proof}

We now show that the solution to the fractional equation \eqref{rand-pde} can be represented as coordinates changed random filed.
 
 Introduce the time dependent random field on the sphere $\mathbb{S}^2_1$,
\begin{align*}
T_t(x) = \sum_{l=0}^\infty \sum_{|m| \leq l} a_{lm} e^{- t \mu_l } \, Y_{lm}(x), \quad x\in \mathbb{S}^2_1,\,\, t\ge 0.
\end{align*}
Let $L$ be the inverse process associated with function $\Phi$ as introduced above, $F$ be the subordinator with the Bern\v{s}tein function $\Psi$. 

Define
$\tau_t = F_{L_t}$ ($F_{L_t} = F \circ L_t$) to be the composition of $F$ and $L$.
\begin{align*}
\mathbf{E}[e^{- \xi \tau_t}] = \mathbf{E}[e^{- \xi \gamma L_t - \Psi(\xi) L_t}] = \widetilde{l}(t, \xi \gamma + \Psi(\xi)), \quad t\geq 0, \quad \xi \geq 0.
\end{align*}
 Let us define the random fields on the sphere $\mathbb{S}^2_1$
\begin{align}
& Y_t(x)= \mathbf{E}[T_{\tau_t}(x) | \mathfrak{F}_T], \label{Ydef}\\
& Z_t(x) = \mathbf{E}[T(x + B_{\tau_t}) | \mathfrak{F}_T] \label{Zdef}
\end{align}
where $\mathfrak{F}_T$ is the $\sigma$-field generated by $T$ and $B$ is a rotational Brownian motion on the sphere $\mathbb{S}^2_1$. Thus, the random field $Y$ turns out to be a time-changed random field whereas, the random field $Z$ is obtained by a random change of the coordinates of $T$. We remark that
\begin{align}
\label{YZrep}
Y_t(x) = \int_0^\infty T_s(x) \, \mathbf{P}(\tau_t \in ds), \quad Z_t(x) = \int_{\mathbb{S}^2_1} T(y) \mathbf{P}_x(B_{\tau_t} \in dy)
\end{align}

\begin{thm}\label{thm2}
Let us consider the solution $X_t(x)$, $x\in \mathbb{S}^2_1$, $t\ge 0$, to equation \eqref{rand-pde} with $\gamma=0$ and the random fields \eqref{Ydef} and \eqref{Zdef}. 
Then the following representations in $L^2(dP \times d\lambda)$ holds true:
\begin{equation}
X_{t}(x)=\mathbf{E}\left[ T(x+B_{F \circ L_t})\big|\mathfrak{F}_{T}\right], \quad t\geq 0 \label{X-rep-F-algebra}
\end{equation}
or equivalently 
\begin{align}
X_t(x) = \mathbf{E}\left[ T_{F \circ L_t}(x)\big|\mathfrak{F}_{T}\right], \quad t\geq 0 \label{X-rep-F-algebrat}
\end{align}
where $\mathfrak{F}_{T}$ is the $\sigma $-field generated by $X_{0}=T$ on $\mathbb{S}^2_1$.
\end{thm}
\begin{proof}
From \eqref{semigB} we have that
\begin{align*}
\mathbf{E} Y_{lm}(x + B_t) = e^{-t \mu_{l}} Y_{lm}(x)
\end{align*}
(see, \cite{Dov} for details)
and therefore
\begin{align*}
\mathbf{E}[T(x + B_{\tau_t}) | \mathfrak{F}_T] 
= & \sum_{lm} a_{lm} \mathbf{E}[Y_{lm}(x+B_{\tau_t}) | \mathfrak{F}_T]\\ 
= & \sum_{lm} a_{lm} \mathbf{E}[e^{-\mu_l \tau_t}] Y_{lm}(x) = \mathbf{E}[T_{\tau_t}(x) | \mathfrak{F}_T]
\end{align*}
that is, the representation \eqref{X-sol-gammaPDE} in $L^2(dP \times d\lambda)$ and the right hand sides of equations \eqref{X-rep-F-algebra} and \eqref{X-rep-F-algebrat} coincide. 
On the other hand, we can write:
\begin{align*}
Z_t(x) = & \mathbf{E}[T(x + B_{\tau_t}) | \mathfrak{F}_T]
= \mathbf{E}\big[\sum_{lm} a_{lm} Y_{lm}(x+B_{\tau_t}) | \mathfrak{F}_T\big]\\
= & \sum_{lm} a_{lm} \mathbf{E}[Y_{lm}(x+B_{\tau_t})]= \sum_{lm} a_{lm} Y_{lm}(x) \mathbf{E}e^{-\mu_l\tau_t}\\
= & \sum_{lm} a_{lm} Y_{lm}(x)\widetilde l (t, \Psi(\mu_l)) = X_t(x).
\end{align*}
In the calculations above we used that $a_{lm}$ are measurable w.r.t. $\mathfrak{F}_T$, $B_{\tau_t}$ is independent of $\mathfrak{F}_T$,
$$\mathbf{E}[Y_{lm}(x+B_{\tau_t})]=Y_{lm}(x) \mathbf{E}e^{-\mu_l\tau_t}$$
(see \cite{Dov}), and
$$
\mathbf{E}e^{-\mu_l\tau_t}=\mathbf{E}e^{-\mu_l F(L_t)}=\mathbf{E}e^{-\Psi(\mu_l) L_t} = \widetilde l (t, \Psi(\mu_l)).
$$
The proof is concluded.
\end{proof}

\begin{remark} In the case where $L_t$ is an inverse stable subordinator, that is, $\Phi(s)=s^\beta$, the derivative $\mathfrak{D}^\Phi_t$ becomes the fractional Caputo-Djrbashian derivative, the Laplace transform $\widetilde{l}(t,\mu)$ is given by the Mittag-Leffler function: $\widetilde{l}(t,\mu)=E_\beta(-t^\beta\mu)$,
 and Theorem 1 above reduces to first part of Theorem 1 in \cite{D}. From our Theorem 2 above it follows that some correction is needed for the second part of Theorem 1 in \cite{D}. Namely,  the representation (3.9) therein should be stated for $\gamma=0$.   
\end{remark}

\begin{remark} One particular case is $\Phi=\Psi$, that is, both space and time derivatives in the equation \eqref{rand-pde} are related to the same Bern\v{s}tein function.
\end{remark}

\begin{remark}
The equation \eqref{rand-pde} can be also considered for more
general functions $\Psi$, not only for Bern\v{s}tein functions. In
particular, we can consider the equation \eqref{rand-pde} with the following fractional diffusion operator
\begin{equation}\label{RieszBessel}
\psi(-\Delta_{\mathbb{S}^2_1}):=(-\Delta_{\mathbb{S}^2_1})^{\alpha/2}(I-\Delta_{\mathbb{S}^2_1})^{\gamma/2},
\end{equation}
where $ \psi(t):=t^{\alpha/2}(1+t)^{\gamma/2}$, and the
representation \eqref{X-sol-gammaPDE} still holds true. Indeed,
the proof relies on two main facts, which are given by the
relations \eqref{eigen-E} and \eqref{eigenPsiDelta}, that is we need to know the eigenfunctions and eigenvalues for the operators. For  the fractional operator \eqref{RieszBessel} the eigenvalues are given by 
$\psi(\mu_l)=\mu_l^{\alpha/2}(1+\mu_l)^{\gamma/2}$ (see
\cite{DL}, p.119--120). Therefore, the representation of the solution of the form \eqref{X-sol-gammaPDE} holds with such $\psi(\mu_l)$ inserted instead of $\Psi(\mu_l)$.

However, if $\Psi$ is a Bern\v{s}tein function, it permits us to
have a deeper insight into the structure of the field $X_t(x)$ and
obtain not only its Karhunen-Lo\`{e}ve expansion, but also have
its representation as a coordinate-changed random field as stated in Theorem 2.

\end{remark}

\begin{remark}
The random field \eqref{X-sol-gammaPDE} obtained as solution to
the fractional Cauchy problem \eqref{rand-pde} can serve to
construct more involved models, in particular, can be used as an
initial condition for fractional SPDE (see, for example,
\cite{ABOW}).
\end{remark}

\begin{example}
Consider the tempered stable subordinator $H$, with the  Bern\v{s}tein function
    \begin{equation}\label{exf}
    \Phi(\lambda)=\left(\lambda+\beta\right)^{\alpha}-\beta^{\alpha}, \quad \alpha\in (0,1), \beta>0.
    \end{equation}
    The corresponding L\'evy measure is given by the formula:
    \begin{equation*}
    {\Pi}(ds)=\frac{1}{\Gamma(1-\alpha)}\alpha e^{-\beta s}s^{-\alpha-1}ds;
    \end{equation*}
    and its tail is
    \begin{equation*}
    \overline{\Pi(s)}=\frac{1}{\Gamma(1-\alpha)}\alpha \beta^{\alpha}\Gamma(-\alpha,s),
    \end{equation*}
    where $\Gamma(-\alpha,s)=\int_{s}^{\infty}e^{-z}z^{-\alpha-1}dz$ is the incomplete Camma function.

    The generalized C-D convolution-type derivative \eqref{D2} for $\Phi$, given by \eqref{exf}, becomes:
    \begin{equation}\label{exfDt}
    \mathfrak{D}^\Phi_t u(t)= \frac{\alpha \beta^{\alpha}}{\Gamma(1-\alpha)}\int_{0}^{t}\frac{\partial }{\partial t }u(t-s)\Gamma(-\alpha,s)ds.
    \end{equation}
    
    We can consider equation \eqref{rand-pde} with such derivative in time, and, correspondingly, 
    in the representation of the solution \eqref{X-sol-gammaPDE} we will have the Laplace transform $\widetilde{l}(t, \lambda)$ of the density  of the inverse tempered stable subordinator, the formula for which is presented, for example, in \cite{AKMS}. As we can see from the results below, $\widetilde{l}(t, \lambda)$ appears also in the expressions for the moments of the fields \eqref{X-sol-gammaPDE}. Therefore, we obtain the model of random field on the sphere with different representation and different properties than that considered in \cite{D}.
\end{example}

We provide the following result concerning the higher-order moments of the solution \eqref{X-sol-gammaPDE}.
\begin{prp}
For $n \in \mathbb{N}$, $\forall\, g \in SO(3)$, the higher-order moments of \eqref{X-sol-gammaPDE} are given by
\begin{equation}
\mathbf{E}[(X_t(g x))^n] = \sum_{l_1=0}^{\infty} \cdots \sum_{l_n =0}^\infty \left( \prod_{j=1}^n \widetilde{l}(t, \gamma+\Psi(\mu_{l_j})) \right) \sqrt{\frac{\prod_{j=1}^n(2l_j+1)}{(4\pi)^n}} \mathbf{E}[a_{l_1 0} \cdots a_{l_n 0}].
\end{equation}
\end{prp}

\begin{proof}
We follow the proof of Proposition 1 in \cite{Dov}. The higher-order moments of \eqref{X-sol-gammaPDE} can be obtained as follows
\begin{align*}
\mathbf{E}[(X_t(x))] =  & \sum_{l_1=0}^{\infty} \cdots \sum_{l_n =0}^\infty \mathbf{E}\left[ \prod_{j=1}^n T_{l_j}(x) \, \widetilde{l}(t, \gamma+\Psi(\mu_{l_j})) \right]
\end{align*}
where
\begin{align*}
\mathbf{E}\left[ \prod_{j=1}^n T_{l_j}(x) \right] = & \sum_{m_1=-l_1}^{+l_1} \cdots \sum_{m_n=-l_n}^{+l_n} \mathbf{E}[a_{l_1m_1} \cdots a_{l_n m_n}] \prod_{j=1}^n Y_{l_j m_j}(x).
\end{align*}
Since the random field $T$ is isotropic, we take advantage of  the property that $T_l(x) \stackrel{law}{=}T(x_N)$ where $x_N=(0,0)$ is the North Pole and that $Y_{lm}(x_N) =0$ for $m \neq 0$ and $Y_{l0}(x_N) = \sqrt{(2l+1)/4\pi}$ (see \cite{Quantum}). We obtain that
\begin{align*}
\mathbf{E}\left[ \prod_{j=1}^n T_{l_j}(x) \right] = & \sqrt{\frac{\prod_{j=1}^n(2l_j+1)}{(4\pi)^n}} \mathbf{E}[a_{l_1 0} \cdots a_{l_n 0}] .
\end{align*}
By collecting all pieces together we get the claimed result. Moreover,  by observing that
\begin{equation*}
\mathbf{E}[(X_t(gx))^n] = \mathbf{E}[(X_t(x))^n], \quad \forall\, g \in SO(3)
\end{equation*}
we complete the proof.
\end{proof}

\subsection{Angular power spectrum}\label{S7}

Under isotropy, the harmonic coefficients $\{a_{lm}\,:\, l\geq 0,\, |m|\leq l \}$ appearing in \eqref{Trep} are such that the power spectrum $\{C_l = \mathbf{E}|a_{lm}|^2\,:\, l\geq 0\}$ associated with the random field $T$ depends uniquely on the frequency $l$. The variance of $T$ can be written as
\begin{equation}
\mathbf{E}[T(x)]^2 = \sum_{l=0}^\infty \frac{2l+1}{4\pi} C_l \quad \textrm{for all }\; x\in \mathbb{S}^2_1 \label{var-intro}
\end{equation}
and thus, the power spectrum must be such that $C_l \sim l^{-\theta}$ as $l \to \infty$ with $\theta > 2$ to ensure $\mathbf{E}[T(x)]^2 < \infty$ as required. As we can see from \eqref{var-intro}, the correlation structure of $T$ is strictly related to the collection $\{C_l\,:\, l\geq 0\}$ of the angular power spectrum.

An interesting review of the characterization of random fields on the sphere $\mathbb{S}^d$, $d\geq 2$, has been recently given in \cite{BingSak}. The authors consider the Karhunen-Lo\`{e}ve expansion in terms of spherical harmonics and Gegenbauer polynomials. Some results about integrability and path-continuity also in the (space) fractional case have been discussed. Many authors focuse on the connection between covariance structure, summability and regularity (see for instance \cite{LangSch} and the references therein). Our result is concerned with subordination of random fields and is stated in the next theorem.

\begin{thm}
For the representation \eqref{X-sol-gammaPDE} of the solution to \eqref{rand-pde} we have that
\begin{itemize}
\item[i)] $X_0 \in \mathcal{H}^s(\mathbb{S}^2_1)$, $s>3$.
\item[ii)] $\forall\, t\geq 0$
\begin{align}
\mathbf{E}[(X_t(x))^2] = \sum_l \frac{2l+1}{4\pi}C_l(t) = \sum_{l} C^*_l(t)
\end{align}
where 
\begin{align}
C^*_l(t) \sim (\overline{\Pi}(t))^2 l^{-\theta - 3}, \quad \theta>2, \quad t>0.
\end{align}
\item[iii)] For the angular power spectrum we also have that, for $t\geq 0$,
\begin{align*}
C_l(t)=\mathbf{E}|a_{lm}(t)|^2 \leq  C_l\, \sup_{\sigma \in (0,1)} \Gamma(1+\sigma)\, (\gamma + \Psi(\mu_l))^{-\sigma}\, \mathbf{E}[(L_t)^{-\sigma}], \quad l\geq 0, \quad |m|\leq l.
\end{align*}
\end{itemize}
\end{thm}
\begin{proof}
For the process $X_t(x)$ introduced in the previous section we have that
\begin{itemize}
\item[i)]
\begin{align*}
\bigg|\left( \gamma - \Psi(-\Delta_{\mathbb{S}^2_1}) \right) X_t(x) \bigg| 
\leq & \sum_{lm} |a_{l,m}| | \gamma + \Psi(\mu_l) | |\widetilde{l}(t, \gamma +\Psi (\mu _{l}))| |Y_{l,m}(x)|\\
\leq & \sum_{lm} |a_{l,m}| | \gamma + \Psi(\mu_l) | |Y_{l,m}(x)|\\
\leq & |\gamma | \sum_{lm} |a_{l,m}|  |Y_{l,m}(x)| +  \sum_{lm} |a_{l,m}|  |\Psi(\mu_l) | |Y_{l,m}(x)|
\end{align*}
and (recall that $|Y_{lm}|\leq l^{1/2}$)
\begin{align*}
\sum_{lm} \mathbf{E}|a_{l,m}|^2  |Y_{l,m}(x)|^2 \leq \sum_l \frac{2l+1}{4\pi}\, l \, C_l
\end{align*}
\begin{align*}
\sum_{lm} \mathbf{E} |a_{l,m}|^2  |\Psi(\mu_l)|^2 |Y_{l,m}(x)|^2 \leq \sum_l \frac{2l+1}{4\pi}  |\Psi(\mu_l)|^2 \, l\, C_l
\end{align*}
thus, from \eqref{def-der-psi-gen} and \eqref{conSumFrac} we obtain that
\begin{align*}
\left( \gamma - \Psi(-\Delta_{\mathbb{S}^2_1}) \right) X_t(x) \in \mathcal{H}^2(\mathbb{S}^2_1) \quad 2s > 6
\end{align*}
\item[ii)]
\begin{align*}
\mathbf{E}[(X_t(x))^2] = \sum_{l \geq 0} \frac{2l+1}{4\pi} \big(\widetilde{l}(t, \gamma+\Psi(\mu_l)) \big)^2  C_l
\end{align*}
is finite only if, $\forall\, t$,
\begin{align*}
\big(\widetilde{l}(t, \gamma+\Psi(\mu_l)) \big)^2  C_l \sim l^{-\theta(\gamma)}, \quad \textrm{with} \quad \theta(\gamma) >2
\end{align*}

The angular power spectrum of $X_t(x)$ can be written as
\begin{align*}
C_l(t) = \big| \mathbf{E}[e^{- (\gamma + \Psi(\mu_l)) L_t}] \big|^2\, C_l
\end{align*}
from which we get the Laplace transform
\begin{align}
\label{AngSpetTemp}
\varphi(\lambda, l) := \int_0^\infty e^{-\lambda t} \sqrt{C_l(t)}\, dt = \frac{\Phi(\lambda)}{\lambda} \frac{\sqrt{C_l}}{\gamma + \Psi(\mu_l) + \Phi(\lambda)}.
\end{align}
If $\sqrt{C_l}/\Psi(\mu_l) \to d_\Psi \geq 0$ as $l\to \infty$ then
\begin{align*}
\varphi(\lambda, l) \to \frac{\Phi(\lambda)}{\lambda}\, d_\Psi = \int_0^\infty e^{-\lambda t} \left(d_\Psi \overline{\Pi}(t) \right) dt \quad \textrm{as} \quad l\to \infty.
\end{align*}
where $\overline{\Pi}$ has been defined in \eqref{tailSymb}. Since $\Psi(\mu_l) \to \infty$ as $l \to \infty$ we get that
\begin{align*}
\varphi(\lambda, l) \sim \frac{\Phi(\lambda)}{\lambda} \frac{\sqrt{C_l}}{\Psi(l^2)} = \int_0^\infty e^{-\lambda t} \left( \frac{\sqrt{C_l}}{\Psi(l^2)}  \, \overline{\Pi}(t) \right)dt.
\end{align*}
Thus, we conclude that, for $t>0$,
\begin{align*}
C_l(t) \sim \left( \frac{\overline{\Pi}(t)}{\Psi(l^2)} \right)^2\, C_l.
\end{align*}
\item[iii)] We notice that
\begin{align*}
\mathbf{E}[e^{-qL_t}] 
= & \int_0^1 \mathbf{P}(e^{-qL_t} > s) ds\\
\leq & \int_0^1 \frac{\mathbf{E}[f(e^{-qL_t})]}{f(s)}ds 
\end{align*}
for $f$ non negative and non decreasing in the set $\{e^{-qL_t}>s\}$.  
By choosing 
$$f(s)=(-\ln s)^{-\sigma}, \quad \sigma\in (0,1), \quad s \in (0,1)$$ 
we obtain
\begin{align*}
\mathbf{E}[e^{-qL_t}] \leq \mathbf{E}[(qL_t)^{-\sigma}] \int_0^1 \frac{ds}{(-\ln s)^{-\sigma}} = q^{-\sigma} \, \Gamma(1 + \sigma) \mathbf{E}[(L_t)^{-\sigma}]
\end{align*}

\end{itemize}
From \eqref{AngSpetTemp} we get the result.
\end{proof}

\begin{remark}
Let us consider the special case $\Phi(\lambda)=\lambda^\beta$ and $\Psi(\xi)=\xi^{\alpha}$. First we note that
\begin{align*}
C_l(t) \leq  & \Gamma(1+\sigma)\, (\gamma + (\mu_l)^\alpha)^{-\sigma}\, \mathbf{E}[(L_t)^{-\sigma}] \, C_l
\end{align*}
Since (see formula \eqref{laplgen})
\begin{align*}
\int_0^\infty e^{-\lambda t} \mathbf{E}[(L_t)^{-\sigma}] \, dt = \frac{\left( \Phi(\lambda) \right)^{\sigma}}{\lambda} \Gamma(1-\sigma) = \Gamma(1-\sigma)\, \lambda^{\beta\sigma - 1}
\end{align*}
we obtain that 
$$\mathbf{E}[(L_t)^{-\sigma}] = \frac{\Gamma(1-\sigma)}{\Gamma(1-\beta \sigma)}\, t^{-\beta \sigma}.$$ 
Thus, $\forall\, t>0$,
\begin{align*}
C_l(t) \leq C_l \, \sup_{\sigma} \frac{\sigma \pi}{\sin \sigma \pi} \, (\gamma + (\mu_l)^\alpha)^{-\sigma}\, \frac{t^{-\beta \sigma}}{\Gamma(1-\beta \sigma)}, \quad l \geq 0.
\end{align*}
\end{remark}

\begin{remark}
{\bf (High-resolution or High-frequency analysis)} The convergence rate of \eqref{var-intro} depends on $C_l$. In particular, the convergence of \eqref{var-intro} depends on the high-frequency behaviour of $C_l$ and therefore on the high-frequency resolution of $T$. In formula \eqref{Trep}, $T_l(x)= \sum_{|m|\leq l} a_{lm} Y_{lm}(x)$ represents the $l$-th frequency component (or projection into the orthonormal space $L^2(\mathbb{S}^2_1)$ spanned by the spherical harmonics $Y_{lm}$) of $T$ and in real data, we get more and more information (or resolution) as $l$ increases. In physical experiments, when we measure the CMB radiation, the power spectrum of the spherical  random fields $T$ is usually unknown and we are interested in the high-frequency analysis concerning the empirical counterpart of the angular power spectrum (see for example \cite{DomPec-book})
\begin{equation}
\widehat{C_l} = \frac{1}{2l+1}\sum_{m=-l}^{+l} |a_{lm}|^2. \label{estimation-C}
\end{equation}
Thus, we may be interested in the high-frequency consistency of $\{\widehat{C_l}\,:\, l \geq 0\}$ or the high-frequency ergodicity of $T$.

The CMB radiation can be affected by some anisotropies usually divided in primary anisotropy (due to effects which occur at the last scattering surface and before) and secondary anisotropy (due to some other effects such as interactions of the radiation with hot gas). Such anisotropies are principally determined by acoustic oscillations and photon diffusion damping.   Acoustic perturbations of initial density fluctuations in the universe made some regions of space hotter and denser than others. We refer to such differences in temperature and density as CMB anisotropies.

The angular power spectrum of the random fields considered in this work exhibits polynomial and/or exponential behaviour (depending on $\Psi$ and $\Phi$) in the high-frequency (or resolution) analysis and therefore, we introduce a large class of models in which many aspects can be captured, such as the Sachs-Wolfe effect (the predominant source of fluctuations) or the Silk damping effect (also called collisionless damping: anisotropies reduced, universe and CMB radiation more uniform). We provide a probabilistic interpretation of the anisotropies of the CMB radiation and we characterize the class $\mathfrak{D}$ introduced in \cite{DomPec-matphys-2010} by means of the coordinates change of random fields. Thus, we can argue about some connection between high-frequency Gaussianity  and high-frequency ergodicity as stated in \cite[Theorem 9]{DomPec-matphys-2010} arriving at a theoretical framework in which we are able to evaluate the asymptotic performance of any statistical procedure based on \eqref{estimation-C}.
\end{remark}

\end{document}